\documentclass[11pt]{amsart}
\usepackage[T1]{fontenc}
\usepackage{lmodern}
\usepackage{amsmath,amssymb,amsthm,mathtools}
\usepackage{microtype}
\usepackage[margin=1.05in]{geometry}
\usepackage{xurl}
\usepackage[hidelinks]{hyperref}
\newcommand{\Z}{\mathbb Z}
\newcommand{\F}{\mathbb F}
\DeclareMathOperator{\GL}{GL}
\DeclareMathOperator{\St}{St}
\DeclareMathOperator{\id}{id}

\DeclareMathOperator{\Hom}{Hom}

\newtheorem{theorem}{Theorem}[section]
\newtheorem{lemma}[theorem]{Lemma}
\newtheorem{proposition}[theorem]{Proposition}
\newtheorem{corollary}[theorem]{Corollary}
\theoremstyle{definition}
\newtheorem{definition}[theorem]{Definition}
\theoremstyle{remark}
\newtheorem{remark}[theorem]{Remark}
\numberwithin{equation}{section}

\newcommand{\eqnum}[1][]{%
  \refstepcounter{equation}%
  \if\relax\detokenize{#1}\relax\else\label{#1}\fi
  \leqno\hbox{\normalfont(\theequation)}}
\newcounter{displaytag}

\newcommand{\eqrownum}[1][]{%
  \hbox{\refstepcounter{equation}%
    \if\relax\detokenize{#1}\relax\else\label{#1}\fi
    \normalfont(\theequation)}}
\makeatletter
\newcommand{\eqrowtag}[2]{%
  \hbox{\refstepcounter{displaytag}%
    \def\@currentlabel{#1}%
    \label{#2}%
    \normalfont(#1)}}
\makeatother

\title[General linear and Steinberg groups over the Leavitt algebra $L_{\mathbb F_2}(1,2)$]{General linear and Steinberg groups\\over the Leavitt algebra $L_{\mathbb F_2}(1,2)$}
\author{Huynh Viet Khanh}
\hypersetup{pdfauthor={Huynh Viet Khanh},pdftitle={General linear and Steinberg groups over the Leavitt algebra L\_F2(1,2)}}
\makeatletter
\renewcommand{\@setauthors}{%
  \begin{center}
  \normalfont
  {\large\authors\par}
  \medskip
  {\small Department of Mathematics and Informatics, HCMC University of Education\\
  280 An Duong Vuong Str., Ho Chi Minh City, Vietnam\\
  \href{mailto:khanhhv@hcmue.edu.vn}{\texttt{khanhhv@hcmue.edu.vn}}\par}
  \end{center}%
}
\makeatother
\date{}
\begin{document}
\begin{abstract}
Let $R=L_{\F_2}(1,2)$. We prove that $\GL_r(R)$ is integrally acyclic for every $r\geq1$ and that the canonical map $\St_r(R)\to\GL_r(R)$ is an isomorphism for every $r\geq3$. The unit group $R^\times$ is finitely presented, and we describe an explicit finite presentation. The homology calculation uses leaf coordinates, simultaneous extensions of ordered frames, and finite-field actions on stabilizers. The Steinberg argument lifts relations from a simply connected frame complex and refines coordinates. We also state separate criteria for the two arguments over rings of characteristic two.
\end{abstract}
\maketitle

\section{Introduction}

Let $R=L_{\F_2}(1,2)$ be the binary Leavitt algebra and let $G=R^\times$. Complete binary leaf sets determine isomorphisms $R^r\cong R$, $M_r(R)\cong R$, and $\GL_r(R)\cong G$ for every $r\geq1$. Isomorphisms between Leavitt algebras and their matrix rings were studied by Abrams, \'{A}nh and Pardo \cite{AAP08}. Here we determine the integral homology of these groups and the relations among their elementary matrices at each finite rank.

The stable calculation alone does not settle these questions. By Ara--Brustenga--Corti\~nas \cite[Theorem~7.6]{ABC09}, $K(R)$ is the homotopy cofiber of multiplication by $-1$ on $K(\F_2)$; hence the stable general linear group has zero positive integral homology. The isomorphisms from leaf coordinates, however, differ from the standard stabilization maps $\GL_r(R)\to\GL_{r+1}(R)$, so this vanishing does not by itself imply vanishing at a fixed rank. Likewise, $K_2(R)=0$ does not determine the kernel of a Steinberg map at a specified finite rank.

For $r\geq3$, let $\phi_r:\St_r(R)\to\GL_r(R)$ be the canonical homomorphism defined by $X_{ij}(a)\mapsto I_r+aE_{ij}$. We prove that
$$
\begin{gathered}
H_n(\GL_r(R),\Z)=0\quad\text{for all }n>0\text{ and }r\geq1,\\
\text{and}\qquad\phi_r:\St_r(R)\xrightarrow{\ \cong\ }\GL_r(R)\quad\text{for all }r\geq3.
\end{gathered}
$$
Thus, at each rank $r\geq3$, every invertible matrix is a product of elementary matrices, and every relation among these generators follows from the ordinary Steinberg relations at that rank. The theorem of Krsti\'c and McCool applied at rank five then implies that $G$ is finitely presented. Appendix~\ref{app:explicit-presentation} records a finite presentation, with bounds on the lengths of the coefficient words and the images of the generators in $G$.

Earlier work describes generators for these groups. Over a fixed finite field, Freeland considered groups generated by copies of general linear groups over the coefficient field associated with finite leaf sets; he also asked whether alternative definitions, including one using all units, agree \cite[Definitions~4.3.9--4.3.10 and pp.~173--174]{Freeland19}. In the binary case, Khanh and Thanh identify the group generated by these matrices with $G$ and discuss finite presentability in terms of unstable Steinberg kernels \cite[Theorem~4.3 and \S8]{KT26}. For a finite graph, Preusser constructs generating sets for the general linear groupoid of its Leavitt path algebra and hence for its general linear groups, including the unit group \cite[Theorem~4.13]{Preusser26}.

Connections with Higman--Thompson groups provide a separate homological comparison; see Pardo \cite{Pardo11} and Gorazd \cite{Gorazd26}. Szymik and Wahl prove that Thompson's group $V$ is integrally acyclic \cite{SW19}. Li obtains acyclicity criteria for topological full groups of ample groupoids \cite[Corollary~D]{Li25}, while Palmer and Wu prove acyclicity for labelled Thompson groups and twisted Brin--Thompson groups \cite[Theorems~A and~C]{PW25}. The acyclicity of the standard copy of $V$ does not by itself imply that of $G$: this copy is a proper subgroup by Remark~\ref{rem:proper-standard-V}, and the homology of a subgroup does not in general determine that of the ambient group.

In Sections~\ref{sec:leaf-coordinates}--\ref{sec:acyclicity-criterion}, we first represent standard stabilization by an endomorphism $c:G\to G$ in leaf coordinates. Stable vanishing makes $c_*$ locally nilpotent on positive integral homology, while conjugacy between $c$ and $c^2$ implies $c_*^2=c_*$. Hence $c_*=0$. To pass from this assertion about stabilization to vanishing at each rank, we use complexes of completable ordered frames, as in homological stability \cite{vdK80,Suslin84,NS90,RWW17}. We prove a simultaneous version of the Reduction Theorem \cite[Theorem~2.2.11]{AAS17}: a single right multiplier works for a finite family and produces nonzero free complements with bases of prescribed cardinalities. This permits a common extension of every finite family of frames in the required range.

To compute stabilizer homology, we use finite-field weights as in Quillen \cite[\S11, Lemmas~15--16]{Quillen72}; compare \cite[Proposition~2.1 and Theorem~3.1]{Knudson02}, \cite[Proposition~1.10 and Theorem~1.11]{NS90}, and \cite[\S2]{Schlichting17}. The embedded fields need not be central. Right multiplication on frame vectors defines an action on the whole frame complex; on a stabilizer it induces inverse left multiplication on the additive normal subgroup, commuting with the action of the remaining general linear group. The stabilizer spectral sequence then proves integral acyclicity by induction on the homological degree.

The Steinberg comparison in Section~\ref{sec:steinberg} uses simple connectivity of the frame complex in rank four, also obtained from the common extension. Brown's presentation \cite[Theorem~1 and \S3]{Brown84} lifts to the Steinberg group, with a section whose image contains every Steinberg generator. Acyclicity and centrality after one stabilization supply the kernel hypothesis at rank four. Voronetsky's refinement theorem transfers the isomorphism to every rank $r\geq3$ \cite[\S4, Proposition~1]{Voronetsky20}. Section~\ref{sec:finite-presentation} applies Krsti\'c--McCool \cite[Theorem~3]{KM99} to obtain finite presentability. The acyclicity and Steinberg criteria are stated separately with their respective hypotheses. No assertion is made about the Steinberg kernel at rank two.

\section{Leaf coordinates and standard inclusions}
\label{sec:leaf-coordinates}

Modules are right modules, and matrices act on columns from the left. Homology has integral coefficients unless another coefficient module is specified.

We prepare the application of the criterion in Section~\ref{sec:acyclicity-criterion} to the binary Leavitt algebra
$$
R=\F_2\langle e,f,e^*,f^*\rangle/ (e^*e=f^*f=1,\ e^*f=f^*e=0,\ ee^*+ff^*=1), \eqnum[eq:leavitt-algebra]
$$
and write $G_r=\GL_r(R)$ and $G=R^\times$. We use the canonical $\F_2$-linear involution on $R$, interchanging the defining generators $e,f$ with $e^*,f^*$. A positive word is a word in $e,f$, including the empty word. If $\alpha=c_1\cdots c_h$, write $\alpha^*=c_h^*\cdots c_1^*$.

The algebra $R$ is nonzero. On the $\F_2$-vector space with basis the infinite binary words, let $e,f$ prepend the indicated letter. Let $e^*,f^*$ remove a matching first letter and send a word beginning with the other letter to zero. These operators satisfy~\eqref{eq:leavitt-algebra}, with nonzero identity.

For these standard coordinates in Leavitt algebras, see \cite{AAP08,AAS17}. A complete finite binary leaf set is obtained from the empty word by repeatedly replacing a word $\alpha$ by $\alpha e,\alpha f$; its leaves may be ordered arbitrarily. The following lemma records the inverse coordinate maps.

\begin{lemma}\label{lem:leaf-coordinates}
For a complete ordered leaf set $(\alpha_1,\ldots,\alpha_m)$, let $L=(\alpha_1,\ldots,\alpha_m)$ and $L^*=(\alpha_1^*,\ldots,\alpha_m^*)^{\mathrm t}$. Then
$$
LL^*=1\qquad\text{and}\qquad L^*L=I_m. \eqnum[eq:leaf-identities]
$$
Consequently $x\mapsto Lx$ and $a\mapsto L^*a$ are mutually inverse maps of right modules $R^m\rightleftarrows R$, and
$$
M_m(R)\rightleftarrows R,\qquad Z\longmapsto LZL^*\qquad\text{and}\qquad a\longmapsto L^*aL \eqnum[eq:leaf-ring-isomorphisms]
$$
are mutually inverse unital ring isomorphisms. Every positive integer $m$ occurs as the size of a complete leaf set.
\end{lemma}

\begin{proof}
Distinct leaves are prefix-incomparable. Cancellation therefore implies $\alpha_i^*\alpha_j=\delta_{ij}$. Starting with $1=1\cdot1^*$, each leaf replacement preserves the sum of the range projections, because $\alpha\alpha^*=\alpha ee^*\alpha^*+\alpha ff^*\alpha^*$. Hence $\sum_i\alpha_i\alpha_i^*=1$, proving~\eqref{eq:leaf-identities}. These identities verify the two pairs of inverse maps and show that the ring maps preserve products and units. A complete leaf set of size $m$ is the comb
$$
L_m=(e^{m-1},e^{m-2}f,\ldots,ef,f),\qquad\text{with }L_1=(1).
$$
Its recurrence is $L_{m+1}=(eL_m,f)$.
\end{proof}

\begin{remark}\label{rem:proper-standard-V}
The standard copy of Thompson's group $V$ consists of the units $\sum_i\alpha_i\beta_i^*$ associated with bijections between complete ordered leaf sets of the same size. These units are unitary. In contrast, $u=1+ef^*$ satisfies $u^{-1}=u$ and $u^*=1+fe^*\ne u$, since $e^*(ef^*-fe^*)f=1$. Thus the standard copy of $V$ is a proper subgroup of $G$.
\end{remark}

The following proposition combines the stable calculation of \cite{ABC09} with an identity involving finite sums, as in the theory of sum rings \cite[\S2]{Wagoner72}.

\begin{proposition}\label{prop:zero-padding}
For every $r\geq1$, the standard inclusion $G_r\to G_{r+1}$, $g\mapsto\operatorname{diag}(g,1)$, induces zero on $H_n(-,\Z)$ for every $n>0$.
\end{proposition}

\begin{proof}
For the graph with one vertex and two loops, the algebraic $K$-theory sequence of Ara--Brustenga--Corti\~nas \cite[Theorem~7.6]{ABC09} has map $1-2=-1$ on $K(\F_2)$. The graph is row-finite and has no sinks. The coefficient field is regular supercoherent: every polynomial ring in finitely many variables over $\F_2$ is regular Noetherian. By the theorem, $K(R)$ is the homotopy cofiber of an equivalence. Hence $K_i(R)=0$ for all $i>0$.

Let $\GL_\infty(R)$ be the direct limit under standard inclusions. The plus construction $B\GL_\infty(R)^+$ has fundamental group $K_1(R)$ and higher homotopy groups $K_i(R)$, $i\geq2$, and preserves integral homology \cite[Chapter~IV, Definitions~1.1 and~1.1.1]{Weibel13}. It is a connected CW complex with trivial homotopy groups, hence is contractible by Whitehead's theorem. Thus
$$
H_n(\GL_\infty(R),\Z)=0\qquad\text{for }n>0. \eqnum[eq:stable-homology-zero]
$$

Define $c:G\to G$ by $c(u)=eue^*+ff^*$. Orthogonality implies $c(u)c(v)=c(uv)$, $c(1)=1$, and $c(u)^{-1}=c(u^{-1})$. For $s\geq0$, let $\Phi_s:R\to M_{2^s}(R)$ have entries $\Phi_s(a)_{\alpha,\beta}=\alpha^*a\beta$, where $\alpha,\beta$ range over the words of length $s$. These are the ring isomorphisms of Lemma~\ref{lem:leaf-coordinates}. Ordering the words at depth $s+1$ by $e\alpha$ followed by $f\alpha$, we have
$$
\Phi_{s+1}(c(u))=\operatorname{diag}(\Phi_s(u),I_{2^s}). \eqnum[eq:dyadic-compression]
$$

Thus the system $(G,c)$ is isomorphic to the dyadic cofinal subsystem of the standard general linear system. The normalized bar complex commutes with this filtered direct limit, since every generator and every finite chain occurs at a finite stage. According to~\eqref{eq:stable-homology-zero}, $\varinjlim(H_n(G,\Z),c_*)=0$. Therefore every element of $H_n(G,\Z)$ is killed by a positive iterate of $c_*$; that is, $c_*$ is locally nilpotent.

It remains to prove that $c_*$ is idempotent. The units associated with the complete leaf sets $(e,fe,ff)$ and $(ee,ef,f)$ are
$$
w=ee\,e^*+ef(fe)^*+f(ff)^*\qquad\text{and}\qquad w^{-1}=e(ee)^*+fe(ef)^*+ff\,f^*.
$$
Both inverse products follow from~\eqref{eq:leaf-identities}. Moreover $we=ee$ and $w(ff^*)w^{-1}=ef(ef)^*+ff^*$. Hence
$$
wc(u)w^{-1}=ee\,u(ee)^*+ef(ef)^*+ff^*=c(c(u)). \eqnum[eq:compression-idempotence]
$$
Inner conjugation acts trivially on homology, so $c_*^2=c_*$. Since $c_*$ is also locally nilpotent, it follows that $c_*=0$ on every positive homology group.

For $Z\in G_r$, the comb recurrence takes the form
$$
L_{r+1}\operatorname{diag}(Z,1)L_{r+1}^* =c(L_rZL_r^*). \eqnum[eq:padding-compression]
$$
Under the leaf-coordinate group isomorphisms, this identity identifies standard stabilization with $c$. Since $c_*=0$, the standard inclusion induces zero on positive homology.
\end{proof}

The map $c$ used here is a group homomorphism. Its affine formula sends $0$ to $ff^*\neq0$, so it is not a ring homomorphism. Proposition~\ref{prop:zero-padding} concerns the specified standard inclusions; the coordinate isomorphisms of Lemma~\ref{lem:leaf-coordinates} induce homology isomorphisms.

\section{Ordered frames and simultaneous extensions}
\label{sec:common-cones}

We use completable ordered frames over nonzero unital $\F_2$-algebras; compare the special unimodular frames of \cite[\S2]{NS90} and \cite[\S2]{Suslin84}. Since a free module may have bases of different finite cardinalities, we specify the number of basis vectors in each complement.

\begin{definition}\label{def:ordered-frames}
Let $A$ be such an algebra and write $\Gamma_r=\GL_r(A)$. The ordered frame semisimplicial set $X_r(A)$ has as its $p$-simplices the tuples $(v_1,\ldots,v_k)$, where $k=p+1\leq r$, admitting a decomposition
$$
A^r=v_1A\oplus\cdots\oplus v_kA\oplus C,\qquad\text{with }C\cong A^{r-k}. \eqnum[eq:frame-decomposition]
$$
Each map $A\to v_iA$, $a\mapsto v_i a$, must be an isomorphism. Since $A\neq0$, the condition $C\cong A^{r-k}$ implies that $C\neq0$ when $k<r$ and $C=0$ when $k=r$. The simplex records the ordered vectors and the formal integer $r-k$; the complement and its basis are not recorded. Faces delete vectors, and there are no simplices in dimensions $p\geq r$.
\end{definition}

Deleting a vector adds its summand to the complement, so the faces in Definition~\ref{def:ordered-frames} are well defined. The condition $C\cong A^{r-k}$ specifies a basis cardinality, which need not be an invariant of the module $C$.

To prove the homology vanishing required in Theorem~\ref{thm:acyclicity-criterion}, we extend each finite family of frames by a common transverse vector; compare \cite[\S2.7, Remark~4]{vdK80} and \cite[\S2, Lemmas~2.1--2.2]{Suslin84}. The Reduction Theorem for Leavitt path algebras treats one nonzero element \cite[Lemma~2.2.8, Theorem~2.2.11 and Corollary~2.2.12]{AAS17}. For $R$, we prove below a simultaneous version: one right multiplier works for a finite family, and the resulting left inverses have nonzero kernels with explicit finite bases.

If $\mu_1,\ldots,\mu_s$ are distinct positive words of lengths at most $D$ and $M>D$, then the words $\mu_i e^M f$ are pairwise prefix-incomparable. Indeed, for comparable original words write $\mu_j=\mu_i\delta$ with $\delta$ nonempty. If $\delta$ contains $f$, its first $f$ occurs before position $M+1$; if $\delta=e^a$, the first $f$ in $\delta e^M f$ occurs at position $M+a+1$. In either case $e^M f$ and $\delta e^M f$ disagree before the shorter word ends. Original words that were prefix-incomparable remain so. Therefore
$$
(\mu_i e^M f)^*(\mu_j e^M f)=\delta_{ij}. \eqnum[eq:word-separation]
$$
In particular, a nonempty sum of distinct positive words is nonzero in $R$: append $e^M f$ and multiply on the left by the star of one of the resulting words to obtain $1$.

\begin{lemma}\label{lem:word-multiplier}
Given finitely many nonzero $a_1,\ldots,a_s\in R$, there are $x\in R$ and nonempty positive words $\eta_1,\ldots,\eta_s$ such that $\eta_i^*a_ix=1$ for every $i$. Each decomposition $R=a_ixR\oplus\ker\eta_i^*$ has an explicit nonzero complement with a finite basis consisting of positive words.
\end{lemma}

\begin{proof}
Every element of $R$ is a finite sum of monomials $\alpha\beta^*$. To obtain this form, cancel every occurrence of a starred letter followed by an unstarred letter. Choose expressions of this form for the $a_i$, and choose $N$ at least as large as all lengths $|\beta|$ occurring in these expressions. Let $\gamma_1,\ldots,\gamma_t$ be all words of length $N$. Each $a_i\gamma_j$ is then a polynomial in positive words. For every $i$, at least one is nonzero, since $a_i=\sum_j(a_i\gamma_j)\gamma_j^*$, although the useful index $j$ may depend on $i$. Collect identical words over $\F_2$ in all these polynomials, and let $L$ bound the lengths of the remaining words.

Put $t_j=e^{j(L+1)}f$ and $x_0=\sum_j\gamma_jt_j$. The words contributed by $(a_i\gamma_j)t_j$ have lengths in $[j(L+1)+1,j(L+1)+1+L]$. These intervals are disjoint for distinct $j$, and concatenation by a fixed $t_j$ preserves distinct words within each interval. Thus $P_i=a_ix_0$ is a nonzero finite sum of distinct words in $e,f$ for every $i$.

We next make the support of each $P_i$ prefix-incomparable. Choose $M$ greater than every word length occurring in the $P_i$ and set $x=x_0e^Mf$. Equation~\eqref{eq:word-separation} makes the support of each $P_ie^Mf$ a prefix antichain. Choose $\eta_i$ in this support. Its coefficient is $1$, so $\eta_i^*a_ix=1$ by~\eqref{eq:word-separation}. Every $\eta_i$ is nonempty.

For a word $\eta=c_1\cdots c_h$, $h\geq1$, let $s_\ell=c_1\cdots c_{\ell-1}c'_\ell$, where $c'_\ell$ is the other binary letter. The words $\eta,s_1,\ldots,s_h$ form a complete leaf set. Hence
$$
\ker\eta^*=\bigoplus_{\ell=1}^h s_\ell R, \eqnum[eq:sibling-kernel]
$$
with inverse coordinate maps $(r_\ell)_\ell\mapsto\sum_\ell s_\ell r_\ell$ and $z\mapsto(s_\ell^*z)_\ell$. If $y=a_ix$ and $b=\eta_i^*$, then $by=1$, and
$$
R\oplus\ker b\rightleftarrows R,\qquad (t,z)\longmapsto yt+z\qquad\text{and}\qquad u\longmapsto(bu,u-ybu)
$$
are mutually inverse. The kernel in~\eqref{eq:sibling-kernel} is nonzero because $h\geq1$ and $s_\ell^*s_\ell=1$.
\end{proof}

\begin{proposition}\label{prop:frame-cone}
For a finite family of simplices of $X_r(R)$, each having at most $r-2$ frame vectors, there is a single vector $v$ extending every frame in the family to a simplex of $X_r(R)$. Each extended frame has a nonzero complement with an explicit basis of the cardinality required in Definition~\ref{def:ordered-frames}.
\end{proposition}

\begin{proof}
For an empty family, choose any standard vertex. Otherwise, let $W_i$ be the span of the $i$th frame, with its given ordered basis, let $k_i$ be its length, and put $q_i=r-k_i\geq2$.

Choose a completion $R^r=W_i\oplus C_i^0$ with $C_i^0\cong R^{q_i}$. Compose the quotient coordinate map with $L_{q_i}:R^{q_i}\to R$ to obtain a surjection $\rho_i:R^r\to R$ with kernel $W_i$. Let $\sigma_i:R\to R^r$ be the corresponding section with image $C_i^0$, so that $\rho_i\sigma_i=1$.

Fix $\Phi=L_r^*:R\to R^r$. The endomorphism $\rho_i\Phi$ of the right module $R$ is left multiplication by an element $a_i$. Since $\rho_i\Phi$ is surjective and $R\ne0$, one has $a_i\ne0$. Apply Lemma~\ref{lem:word-multiplier} to the $a_i$, and put $v=\Phi x$, $y_i=\rho_i v=a_ix$, and $b_i=\eta_i^*$. Since $b_iy_i=1$, the map $R\to R^r$ defined by $a\mapsto va$ is a split injection, with left inverse $b_i\rho_i$ for every $i$.

Put $w_i=v-\sigma_i y_i\in W_i$. We have
$$
R^r=W_i\oplus vR\oplus\sigma_i\ker b_i, \eqnum[eq:cone-decomposition]
$$
with coordinate map and inverse
$$
\begin{aligned}
 (w,t,z)&\longmapsto w+vt+\sigma_i z\qquad\text{and}\\
 u&\longmapsto\begin{pmatrix}
 (I-\sigma_i\rho_i)u-w_i b_i\rho_i u\\
 b_i\rho_i u\\
 \rho_i u-y_i b_i\rho_i u
 \end{pmatrix}.
 \end{aligned} \eqnum[eq:cone-coordinate-inverses]
$$
The first inverse component lies in $W_i$ because $\rho_i\sigma_i=1$ and $\rho_iw_i=0$; the last lies in $\ker b_i$ because $b_iy_i=1$. If $u=w+vt+\sigma_i z$, then $\rho_i u=y_it+z$ and $b_i\rho_i u=t$, so the inverse recovers $(w,t,z)$. Conversely, substituting the displayed components into the forward map returns $u+(v-w_i-\sigma_i y_i)b_i\rho_i u=u$.

Let $h_i=|\eta_i|$, and let $S_i=(s_1,\ldots,s_{h_i})$ be the row in~\eqref{eq:sibling-kernel}. The map
$$
\sigma_iS_iL_{h_i}^*L_{q_i-1}:R^{q_i-1} \longrightarrow\sigma_i\ker b_i \eqnum[eq:complement-formal-coordinates]
$$
is an isomorphism, with inverse on this complement $L_{q_i-1}^*L_{h_i}S_i^*\rho_i$. By $\rho_i\sigma_i=1$, $S_i^*S_i=I_{h_i}$, and~\eqref{eq:leaf-identities}, the inverse product on $R^{q_i-1}$ is the identity. The other product is $\sigma_iS_iS_i^*\rho_i$, which is the identity on $\sigma_i\ker b_i$ by~\eqref{eq:sibling-kernel}. Since $q_i-1\geq1$ and $h_i\geq1$, the complement is nonzero. Thus~\eqref{eq:cone-decomposition} extends each specified frame to a simplex of $X_r(R)$.
\end{proof}

\begin{corollary}\label{cor:frame-homology}
For $r\geq3$, one has $\widetilde H_d(X_r(R),\Z)=0$ for $0\leq d\leq r-3$.
\end{corollary}

\begin{proof}
Let $z$ be a finite ordered integral $d$-cycle, using the reduced chain complex when $d=0$. Its support, together with all faces of its simplices, is finite, and every frame in this family has at most $d+1\leq r-2$ vectors. Choose a common extension $v$ by Proposition~\ref{prop:frame-cone}. Prepending $v$ defines a chain operator $h_v$ on this finite family, and the ordered boundary satisfies $\partial h_v+h_v\partial=\id$. At the augmentation, set $h_v(1)=[v]$; the signs are the usual integral alternating signs. Since $\partial z=0$, one has $\partial h_vz=z$. The same augmented calculation for reduced zero-cycles proves connectedness.
\end{proof}

\begin{corollary}\label{cor:frame-simple-connectivity}
The realization $|X_4(R)|$ is connected and simply connected.
\end{corollary}

\begin{proof}
By Proposition~\ref{prop:frame-cone}, any two vertices admit a common apex and are therefore joined by a path, so $|X_4(R)|$ is connected. Let a finite edge loop be given. Choose one common apex $v$ for all its vertices and ordered edges, with complements for the resulting frames. For each ordered edge $(u,w)$ traversed by the loop, use the triangle $(v,u,w)$. Adjacent triangles share their radial edges and therefore form a continuous triangular fan whose boundary is the given loop. A reverse traversal uses the same ordered edge with the opposite orientation; it does not replace that edge by the distinct cell $(w,u)$. Repeated vertices cause no difficulty for the resulting map of a disk. By cellular approximation, every element of the fundamental group is represented by a finite edge loop, and the fan contracts it. Thus $|X_4(R)|$ is simply connected.
\end{proof}

\section{An acyclicity criterion and its application}
\label{sec:acyclicity-criterion}

Throughout this section, $A$ is a nonzero unital $\F_2$-algebra, $\Gamma_r=\GL_r(A)$, and $X_r(A)$ is the ordered frame semisimplicial set of Definition~\ref{def:ordered-frames}.

\begin{theorem}\label{thm:acyclicity-criterion}
Suppose that $A^2\cong A$ as right $A$-modules. Assume also that:
\begin{enumerate}
\item for every $n>0$, the standard inclusion $\Gamma_{n+2}\to\Gamma_{n+3}$, $g\mapsto\operatorname{diag}(g,1)$, induces zero on $H_n(-,\Z)$;
\item $\widetilde H_i(X_r(A),\Z)=0$ for $0\leq i\leq r-3$ and $r\geq4$.
\end{enumerate}
Then $H_n(A^\times,\Z)=0$ for every $n>0$.
\end{theorem}

To control integral homology and the coefficient modules below, we transfer Quillen's scalar weight calculation through reduction modulo two; see \cite[\S11, Lemmas~15--16]{Quillen72}.

\begin{lemma}\label{lem:scalar-homology}
Let $m\geq1$, let $F=\F_{2^m}$ and $C=F^\times$, and regard an $F$-vector space $V$ as a discrete additive group. For $1\leq j<m$, the group $H_j(V,\Z)$ is annihilated by $2$, and the norm element $N_C=\sum_{c\in C}[c]\in\Z[C]$ acts as zero on it. If a group $H$ acts $F$-linearly on $V$, then $N_C$ acts as zero on $H_i(H,H_j(V,\Z))$ for every $i\geq0$. These homology groups and all their $C$-stable subquotients have zero $C$-coinvariants. The same conclusions hold for inverse scalar multiplication.
\end{lemma}

\begin{proof}
Suppose first that $V$ is finite dimensional over $F$. Its additive group is a finite product of copies of $C_2$. The integral cellular chain complex of $BC_2=\mathbb{RP}^{\infty}$ has one copy of $\Z$ in each nonnegative degree, with differential $2$ in positive even degrees and zero in odd degrees. It is the direct sum of $\Z[0]$ and two-term complexes $\Z\xrightarrow{2}\Z$ in degrees $(2a,2a-1)$, $a\geq1$. Multiplication by $2$ on each two-term complex is null-homotopic: take the identity map from degree $2a-1$ to degree $2a$ as the homotopy.

The product chain complex is the tensor product of these complexes. Every summand other than $\Z[0]$ contains a two-term factor. Applying the preceding homotopy to one such factor, with its tensor sign, makes multiplication by $2$ null-homotopic on that summand; the cross terms cancel. Consequently $2H_j(V,\Z)=0$ for $j>0$. In the long exact sequence associated with $\Z\xrightarrow{2}\Z\to\F_2$, the kernel of reduction modulo two is the image of multiplication by $2$. Thus reduction induces a natural injection
$$
H_j(V,\Z)\longrightarrow H_j(V,\F_2)\qquad\text{for }j>0, \eqnum[eq:integral-mod-two]
$$
which commutes with every automorphism of $V$.

The mod-two cohomology of $V$ is the polynomial algebra on $H^1(V,\F_2)=\Hom_{\F_2}(V,\F_2)$. Indeed, $H^*(C_2,\F_2)=\F_2[t]$ with $|t|=1$, and this description extends to a finite product by the field K\"unneth formula. Since the algebra is generated by degree one, the description is natural under all linear automorphisms of $V$.

After extending scalars to a splitting field, the eigencharacters of scalar multiplication on $V$ are $\lambda\mapsto\lambda^{2^a}$, $0\leq a<m$, each with multiplicity $\dim_F V$. They may also be read from the isomorphism
$$
F\otimes_{\F_2}F\longrightarrow\prod_{a=0}^{m-1}F,\qquad x\otimes y\longmapsto(x^{2^a}y)_a.
$$
With the contragredient convention the cohomology generators have weights $-2^a$. Dualizing each finite cohomological degree shows that the weights on $H_j(V,\F_2)$ are sums of exactly $j$ members of $\{1,2,4,\ldots,2^{m-1}\}$, with repetitions, modulo $d=2^m-1$.

Quillen's counting lemma \cite[Lemma~16]{Quillen72}, specialized to $p=2$, shows that no such sum is zero modulo $d$ when $0<j<m$. Its cyclic carrying argument is as follows. Place one counter for each summand in the corresponding cyclic position of $\Z/m\Z$. Replace two counters in one position by one counter in the next. Each replacement preserves the residue modulo $d$, including the replacement in the last position, and strictly decreases the positive number of counters. The process terminates at a nonempty set of distinct positions with fewer than $m$ counters. Its ordinary sum lies between $1$ and $d-1$, so its residue is nonzero.

The order $d$ is odd, and the norm vanishes on every nontrivial character in characteristic two. By the weight calculation, $N_C$ acts as zero on $H_j(V,\F_2)$ for $1\leq j<m$, and hence on $H_j(V,\Z)$ by~\eqref{eq:integral-mod-two}. Inverse scalar multiplication replaces every weight by its negative and leaves the conclusion unchanged.

For arbitrary $V$, its finite-dimensional $F$-subspaces form a directed system of $C$-stable subgroups. Every finite bar chain is contained in one of them. Since filtered colimits of abelian groups are exact, the bar complex shows that group homology commutes with this union. The exponent and norm assertions therefore hold without a finiteness assumption on $V$.

If $H$ acts $F$-linearly, its action commutes with $C$. On the bar complex computing $H_i(H,H_j(V,\Z))$, the norm acts coefficientwise as zero, and hence induces zero on homology. These homology groups are annihilated by $2$, and both properties pass to $C$-stable subquotients. On coinvariants the norm acts as multiplication by the odd integer $d$, hence as the identity on a group annihilated by $2$. The coinvariants are therefore zero.
\end{proof}

If $C$ has odd order $d$ and $M$ is a $2$-primary $C$-module, multiplication by $d$ is invertible on $M$. The averaging operator $d^{-1}N_C$ identifies invariants and coinvariants and makes both functors exact on this category. Thus, if $M_C=0$, every $C$-stable subquotient of $M$ has zero coinvariants. For a module annihilated by $2^a$, the inverse of $d$ may be represented by an integer inverse modulo $2^a$. We will use this exactness for a kernel whose graded pieces are annihilated by $2$ but whose exponent is bounded by $2^n$.

Assume now that $A^2\cong A$. Then $A^m\cong A$ for every $m\geq1$, by induction using $A^{m+1}\cong A^m\oplus A$. Represent one such isomorphism by a row $L$ and its inverse by a column $L'$, so that $LL'=1$ and $L'L=I_m$. Then $T\mapsto LTL'$ and $a\mapsto L'aL$ are mutually inverse unital ring isomorphisms between $M_m(A)$ and $A$, and hence $\Gamma_m\cong A^\times$.

Composing the regular representation $\F_{2^m}\to M_m(\F_2)$ with the entrywise inclusion $M_m(\F_2)\hookrightarrow M_m(A)$ and the isomorphism $M_m(A)\to A$ embeds $\F_{2^m}$ into $A$. Indeed, the composite is unital and therefore injective because $A\neq0$.

For $k,q\geq1$, put
$$
J_{k,q}(A)=
 \left\{\begin{pmatrix}I_k&B\\0&H\end{pmatrix}:
 B\in M_{k,q}(A)\text{ and }H\in\Gamma_q\right\}. \eqnum[eq:frame-stabilizer]
$$
Write $U=M_{k,q}(A)_{\mathrm{add}}$. The extension $1\to U\to J_{k,q}(A)\xrightarrow{\pi}\Gamma_q\to1$ splits by $s(H)=\operatorname{diag}(I_k,H)$. If $u(B)=\left(\begin{smallmatrix}I_k&B\\0&I_q\end{smallmatrix}\right)$, then $s(H)u(B)s(H)^{-1}=u(BH^{-1})$. Thus the left $\Gamma_q$-action on $U$ is $B\mapsto BH^{-1}$. In particular,
$$
\begin{pmatrix}I_k&B\\0&H\end{pmatrix}
 =u(BH^{-1})s(H)\qquad\text{and}\qquad
 \begin{pmatrix}I_k&B\\0&H\end{pmatrix}^{-1}
 =\begin{pmatrix}I_k&-BH^{-1}\\0&H^{-1}\end{pmatrix}.
$$

Fix $n>0$, choose $m>n$, and embed $F=\F_{2^m}$ into $A$ as above. Let $C=F^\times$. Conjugation by $t_\lambda=\operatorname{diag}(\lambda I_k,I_q)$ fixes $s(\Gamma_q)$ pointwise and sends $B\in U$ to $\lambda B$.

The next calculation is a form of the scalar argument for affine groups; compare \cite[Theorem~1.9]{Suslin84}, \cite[Proposition~1.10, Theorem~1.11 and Remark~1.13]{NS90}, \cite[\S2]{Schlichting17}, and \cite[Proposition~2.1]{Knudson02}. The role of central scalars in the Nesterenko--Suslin argument is noted in \cite[Definition~2.1, footnote~1]{Schlichting17}. Here $F$ need not be central in $A$: left scalar multiplication commutes with the $\Gamma_q$-action, since $\lambda(BH^{-1})=(\lambda B)H^{-1}$.

\begin{lemma}\label{lem:stabilizer-scalars}
For the scalar action just defined, let
$$
Q_n=\ker\bigl(H_n(J_{k,q}(A),\Z)\xrightarrow{\pi_*}H_n(\Gamma_q,\Z)\bigr).
$$
Then $2^nQ_n=0$, $(Q_n)_C=0$, and the projection induces an isomorphism
$$
H_n(J_{k,q}(A),\Z)_C\cong H_n(\Gamma_q,\Z). \eqnum[eq:stabilizer-coinvariants]
$$
If the standard inclusion $\Gamma_q\to\Gamma_{k+q}$, $H\mapsto\operatorname{diag}(H,I_k)$, induces zero on $H_n(-,\Z)$, then the inclusion $J_{k,q}(A)\to\Gamma_{k+q}$ also induces zero on $H_n(-,\Z)$. All these assertions remain valid when the scalar actions are inverted.
\end{lemma}

\begin{proof}
Associated with the split extension is the $C$-equivariant Lyndon--Hochschild--Serre spectral sequence
$$
E^2_{p,j}=H_p(\Gamma_q,H_j(U,\Z)) \ \Longrightarrow\ H_{p+j}(J_{k,q}(A),\Z). \eqnum[eq:stabilizer-lhs]
$$
We use its usual homological indexing and projection edge map; see \cite[Theorem~6.8.2]{Weibel94}. The row $j=0$ survives unchanged to $E^\infty$. Indeed, the section and projection compare it with the spectral sequence of the extension of $\Gamma_q$ with trivial kernel and induce the identity on that row. There are no incoming differentials, and naturality with the section makes the outgoing differentials zero.

Consequently $Q_n=F_{n-1}H_n(J_{k,q}(A),\Z)$, with graded pieces $E^\infty_{p,n-p}$ for $0\leq p<n$. Each is a $C$-stable subquotient of $E^2_{p,n-p}$ in~\eqref{eq:stabilizer-lhs}, where $1\leq n-p\leq n<m$. By Lemma~\ref{lem:scalar-homology}, these pieces are annihilated by $2$ and have zero $C$-coinvariants. There are at most $n$ pieces, so $2^nQ_n=0$. Exactness of coinvariants on $2$-primary modules, applied successively to the filtration, implies $(Q_n)_C=0$.

The section is fixed by $C$ and splits $\pi_*$. Together with $(Q_n)_C=0$, this proves~\eqref{eq:stabilizer-coinvariants}. Let $i:J_{k,q}(A)\to\Gamma_{k+q}$ be the inclusion. If $\alpha_\lambda$ denotes the scalar automorphism of $J_{k,q}(A)$, then $i\alpha_\lambda=c_{t_\lambda}i$, where $c_{t_\lambda}$ is conjugation in $\Gamma_{k+q}$. Since inner automorphisms induce the identity on homology, $i_*$ is $C$-invariant and therefore factors through the coinvariant quotient $H_n(J_{k,q}(A),\Z)_C$. Under~\eqref{eq:stabilizer-coinvariants}, the resulting map is induced by $H\mapsto\operatorname{diag}(I_k,H)$. A block permutation conjugates this map to $H\mapsto\operatorname{diag}(H,I_k)$, the specified standard inclusion, whose map on $H_n$ is zero by hypothesis. Inverting the scalar actions permutes the summands of the norm and preserves the argument.
\end{proof}

\begin{proof}[Proof of Theorem~\ref{thm:acyclicity-criterion}]
Let $\mathcal B_r=E\Gamma_r\times_{\Gamma_r}|X_r(A)|$. For a free right $\Z\Gamma_r$-resolution $P_*$ of $\Z$, the double complex $P_t\otimes_{\Z\Gamma_r}C_p(X_r(A),\Z)$ computes $H_*(\mathcal B_r,\Z)$. The spectral sequence obtained by taking homology first in the frame direction has $E^2_{a,b}=H_a(\Gamma_r,H_b(X_r(A),\Z))$. By the second hypothesis, the projection $\mathcal B_r\to B\Gamma_r$ induces an isomorphism on $H_n$ whenever $r\geq n+3$.

For the other filtration, $\Gamma_r$ acts transitively on the ordered $(k-1)$-simplices: an isomorphism from $A^{r-k}$ onto the complement in~\eqref{eq:frame-decomposition} completes a given frame to an invertible matrix. The stabilizer of the standard $k$-frame is $J_{k,r-k}(A)$, and the stabilizer of a full frame is trivial; put $J_{r,0}(A)=1$. Since the frames are ordered, these stabilizers fix their simplices pointwise. By Shapiro's lemma, the spectral sequence is
$$
E^1_{p,t}=H_t(J_{p+1,r-p-1}(A),\Z) \ \Longrightarrow\ H_{p+t}(\mathcal B_r,\Z), \qquad\text{for }0\leq p\leq r-1; \eqnum[eq:frame-spectral-sequence]
$$
see \cite[\S6.1.15]{Weibel94} and compare \cite[Lemmas~2.3--2.4]{NS90}. All filtrations in a fixed total degree are finite.

The row $t=0$ in~\eqref{eq:frame-spectral-sequence} has one copy of $\Z$ in each column. The differential out of column $p$ is multiplication by $\sum_{i=0}^p(-1)^i$. Thus
$$
E^2_{p,0}=0\qquad\text{for }1\leq p\leq r-2. \eqnum[eq:frame-bottom-row]
$$
For positive even $p$ the outgoing differential is injective, and for odd $p$ the incoming differential from $p+1$ is the identity.

We induct on $n$. Assume $H_t(A^\times,\Z)=0$ for $0<t<n$; when $n=1$ this assumption is empty. Set $r=n+3$. The projection $\mathcal B_r\to B\Gamma_r$ then induces an isomorphism in degree $n$. Choose $m>n$ so that Lemma~\ref{lem:scalar-homology} applies in every positive degree at most $n$, and fix a unital copy of $F=\F_{2^m}$ in $A$. Put $C=F^\times$.

For $\lambda\in C$, define
$$
T_\lambda(v_1,\ldots,v_k)=(v_1\lambda,\ldots,v_k\lambda). \eqnum[eq:global-rescaling]
$$
Right multiplication by $\lambda$ need not be $A$-linear when $F$ is not central, but it preserves every frame: $v_i\lambda A=v_iA$, so the same complement may be used. The maps commute with deletion and satisfy $g(v_i\lambda)=(gv_i)\lambda$ for $g\in\Gamma_r$. Since $C$ is abelian, they define a $C$-action on the ordered semisimplicial set.

The identity on $E\Gamma_r$ together with~\eqref{eq:global-rescaling} defines an action on $\mathcal B_r$ covering the identity of $B\Gamma_r$. On the double complex, $1\otimes T_\lambda$ commutes with both differentials and preserves both filtrations. The Borel projection is therefore $C$-equivariant when the base has the trivial action. Since it induces an isomorphism on $H_n$, the $C$-action on $H_n(\mathcal B_r,\Z)$ is trivial.

To identify the induced action on~\eqref{eq:frame-spectral-sequence}, let $\sigma_k$ be the standard $k$-frame and put $t_{\lambda,k}=\operatorname{diag}(\lambda I_k,I_{r-k})$. Then $T_\lambda\sigma_k=t_{\lambda,k}\sigma_k$, so transport back to $\sigma_k$ by $t_{\lambda,k}^{-1}$ induces the stabilizer automorphism
$$
h\longmapsto t_{\lambda,k}^{-1}ht_{\lambda,k},\qquad\text{that is,}\qquad
 \begin{pmatrix}I_k&B\\0&H\end{pmatrix}
 \longmapsto\begin{pmatrix}I_k&\lambda^{-1}B\\0&H\end{pmatrix}.
$$
Equivalently, the orbit map $gJ\mapsto gt_{\lambda,k}J$ induces $x\mapsto xt_{\lambda,k}$ on $E\Gamma_r/J$, and $(xh)t_{\lambda,k}=(xt_{\lambda,k})(t_{\lambda,k}^{-1}ht_{\lambda,k})$ confirms the conjugation convention. Thus the action on the $E^1$-term is the inverse scalar action of Lemma~\ref{lem:stabilizer-scalars}.

The chosen deletion representatives commute with these stabilizer automorphisms only up to inner conjugation. If the rows of $B$ are $b_1,\ldots,b_k$, deleting row $i$ and placing that coordinate first among the new quotient coordinates defines
$$
f_i\begin{pmatrix}I_k&B\\0&H\end{pmatrix}
 =\begin{pmatrix}
 I_{k-1}&0&B_{\widehat i}\\
 0&1&b_i\\
 0&0&H
 \end{pmatrix}.
$$
The source action multiplies both $B_{\widehat i}$ and $b_i$ by $\lambda^{-1}$, whereas the target action multiplies only $B_{\widehat i}$. Their discrepancy is conjugation by $\operatorname{diag}(I_{k-1},\lambda^{-1},I_q)$ in the target stabilizer, where $q=r-k$. Inner automorphisms induce the identity on homology, and the global action on the double complex ensures compatibility with every page of the spectral sequence.

For $1\leq p<n$, put $t=n-p$ and $q=r-p-1$. Then $0<t<n$ and $q\geq3$. Since $\Gamma_q\cong A^\times$, it follows from the induction hypothesis that $H_t(\Gamma_q,\Z)=0$. By Lemma~\ref{lem:stabilizer-scalars}, with the inverse scalar action, $H_t(J_{p+1,q}(A),\Z)=Q_t$ is $2$-primary and has zero $C$-coinvariants. Each $E^\infty_{p,n-p}$ is a $C$-stable subquotient of $E^1_{p,n-p}$, so exactness of coinvariants on $2$-primary modules implies
$$
\bigl(E^\infty_{p,n-p}\bigr)_C=0\qquad\text{for }1\leq p<n. \eqnum[eq:lower-filtration-coinvariants]
$$
When $n=1$, there are no pieces in this range.

Since $r=n+3$, \eqref{eq:frame-bottom-row} implies $E^\infty_{n,0}=0$. For $p=0$, the stabilizer map and the Borel projection form
$$
H_n(J_{1,r-1}(A),\Z)\longrightarrow H_n(\mathcal B_r,\Z)\xrightarrow{\cong}H_n(\Gamma_r,\Z).
$$
The composite is induced by the stabilizer inclusion and is zero by hypothesis~(1) and Lemma~\ref{lem:stabilizer-scalars}, since $r-1=n+2$. The first map has image $F_0H_n(\mathcal B_r,\Z)$; hence this filtration subgroup is zero.

All graded pieces of the finite filtration of $H_n(\mathcal B_r,\Z)$ now have zero $C$-coinvariants. Applying the right-exact coinvariant functor successively to the short exact sequences of this filtration, we obtain $H_n(\mathcal B_r,\Z)_C=0$. Since the $C$-action on $H_n(\mathcal B_r,\Z)$ is trivial, $H_n(\mathcal B_r,\Z)=0$. The Borel comparison and the isomorphism $\Gamma_r\cong A^\times$ then imply $H_n(A^\times,\Z)=0$, completing the induction.
\end{proof}

\begin{theorem}\label{thm:acyclicity}
The unit group of $R=L_{\F_2}(1,2)$ is integrally acyclic. More generally, $H_n(\GL_r(R),\Z)=0$ for every $n>0$ and $r\geq1$.
\end{theorem}

\begin{proof}
Lemma~\ref{lem:leaf-coordinates}, Proposition~\ref{prop:zero-padding}, and Corollary~\ref{cor:frame-homology} verify the hypotheses of Theorem~\ref{thm:acyclicity-criterion}. It follows that $H_n(R^\times,\Z)=0$ for all $n>0$. The leaf ring isomorphisms~\eqref{eq:leaf-ring-isomorphisms} identify every $\GL_r(R)$ with $R^\times$, so the same conclusion holds in every rank.
\end{proof}

\section{The Steinberg comparison}
\label{sec:steinberg}

The acyclicity of the unit group will be used to verify the following criterion for the Steinberg map.

For a nonzero unital ring $B$, write $t_{ij}(a)=I+aE_{ij}$ and let $E_m(B)$ be the subgroup of $\GL_m(B)$ generated by these matrices. For $m\geq3$, the ordinary Steinberg group $\St_m(B)$ has generators $X_{ij}(a)$, where $i\ne j$ and $a\in B$, and relations
$$
\begin{aligned}
\eqrownum &\quad X_{ij}(a)X_{ij}(b)=X_{ij}(a+b),\\
\eqrownum &\quad [X_{ij}(a),X_{kl}(b)]=1 \quad \text{for }i\ne l\text{ and }j\ne k,\\
\eqrownum &\quad [X_{ij}(a),X_{jk}(b)]=X_{ik}(ab) \quad \text{for distinct }i,j,k.
\end{aligned}
$$
Here $[x,y]=xyx^{-1}y^{-1}$. Let $\phi_m:\St_m(B)\to\GL_m(B)$ send $X_{ij}(a)$ to $t_{ij}(a)$, and put $N_m(B)=\ker\phi_m$. The standard stabilization homomorphism $j_m:\St_m(B)\to\St_{m+1}(B)$ is defined by $j_m(X_{ij}(a))=X_{ij}(a)$.

For a ring $B$ of characteristic two, let $X_n(B)$ be the ordered frame semisimplicial set of Definition~\ref{def:ordered-frames}.

\begin{theorem}[Steinberg comparison criterion]
\label{thm:raw-criterion}
Let $B$ be a nonzero unital ring of characteristic two, and fix $n\geq4$. Suppose that
\begin{enumerate}
\item $\GL_{n-1}(B)=E_{n-1}(B)$ and $\GL_{n-2}(B)=E_{n-2}(B)$;
\item $j_{n-1}(N_{n-1}(B))=1$ in $\St_n(B)$;
\item $|X_n(B)|$ is simply connected.
\end{enumerate}
Then $\phi_n:\St_n(B)\to\GL_n(B)$ is an isomorphism.
\end{theorem}

\begin{proof}
Put $G=\GL_n(B)$, $S=\St_n(B)$, and let $b_1,\ldots,b_n$ be the standard columns.

The action of $G$ on ordered $k$-frames is transitive, since a basis of a permitted complement completes the frame to an automorphism of $B^n$. In particular, there is one orbit in each of dimensions zero, one, and two. The vertex stabilizer $J$ and the stabilizer $K$ of the ordered edge are
$$
J=\left\{j(b,H)=\begin{pmatrix}1&b\\0&H\end{pmatrix}\right\}\qquad\text{and}\qquad
 K=\left\{k(a,b,H)=
 \begin{pmatrix}1&0&a\\0&1&b\\0&0&H\end{pmatrix}\right\}.
$$
In $J$, the row $b$ has length $n-1$ and $H\in\GL_{n-1}(B)$. In $K$, the rows have length $n-2$ and $H\in\GL_{n-2}(B)$. The upper rows in these full stabilizers are unrestricted. Let $\tau=(12)$ and $h=(23)$ be permutation matrices, so that $h\in J$, and put $\eta(k)=\tau k\tau^{-1}$. This automorphism of $K$ interchanges its two upper rows.

We apply Brown's presentation for a group acting on a simply connected complex \cite[Theorem~1 and \S3]{Brown84}. For the ordered frame complex, the Borel construction $Y=EG\times_G|X_n(B)|$ has fundamental group
$$
\Pi=\langle J,T\mid TkT^{-1}=\eta(k)\text{ for }k\in K,\quad ThT=hTh\rangle. \eqnum[eq:borel-presentation]
$$
This notation includes every multiplication relation in $J$. To see that the remaining relations are complete, consider the cellular Borel construction through total dimension two. The vertex orbit contributes $BJ$. In the cylinder over $BK$, the product of the zero-cell of $BK$ with the edge has total dimension one and supplies $T$; the products of the one-cells of $BK$ with the edge impose the conjugation relations between the endpoint inclusions. The zero-cell in the classifying space of the triangle stabilizer contributes one further two-cell. All other cells have total dimension at least three.

To determine the triangle relation, identify $\pi_1(Y)$ with the group of pairs $(g,[p])$, where $p$ is a path from $b_1$ to $gb_1$ taken up to homotopy relative to its endpoints, with multiplication
$$
(g,[p])(g',[p'])=(gg',[p*(gp')]).
$$
Represent $j\in J$ by its constant path and $T$ by $(\tau,E)$, where $E$ is the ordered edge $(b_1,b_2)$. For $k\in K$, both $Tk$ and $\eta(k)T$ have projection $\tau k$ and path $E$. The word $ThT$ has projection $(13)$ and path $b_1\to b_2\to b_3$, whereas $hTh$ has the same projection and path $b_1\to b_3$. The boundary of the ordered triangle $(b_1,b_2,b_3)$ therefore imposes the second relation in \eqref{eq:borel-presentation}.

The cells $(u,v)$ and $(v,u)$ are distinct; traversing one in reverse does not replace it by the other. Since $h^2=1$, the braid relation implies $(hT)h(hT)^{-1}=T$ and hence $T^2=1$.

The homotopy sequence of $|X_n(B)|\to Y\to BG$, together with the third hypothesis, shows that the projection $p:\Pi\to G$, defined by $p(j)=j$ and $p(T)=\tau$, is an isomorphism.

We now lift this presentation to $S$. Permuting indices is an automorphism of the Steinberg presentation, so the second hypothesis also applies to stabilization on coordinates $2,\ldots,n$. For $H\in\GL_{n-1}(B)=E_{n-1}(B)$, choose an elementary word representing $H$ and evaluate it on these coordinates. If two words represent the same $H$, their quotient is a word in $N_{n-1}(B)$; the second hypothesis makes its value in $S$ trivial. Hence the evaluation is independent of the chosen word, and concatenation defines a homomorphism
$$
\ell:\GL_{n-1}(B)\longrightarrow S
$$
covering $H\mapsto\operatorname{diag}(1,H)$.

For a row $b=(b_2,\ldots,b_n)$ set $x_1(b)=\prod_{j=2}^nX_{1j}(b_j)$, in increasing order. The factors commute. According to the Steinberg relations,
$$
\ell(H)x_1(b)\ell(H)^{-1}=x_1(bH^{-1}). \eqnum[eq:row-action]
$$
Indeed, for $H=t_{ij}(c)$ with $i,j\geq2$, conjugation changes only the $j$-th entry, replacing it by $b_j-b_i c$. This is right multiplication by $I-cE_{ij}$; the third Steinberg relation produces the new term, while all other factors commute. Elementary generation extends the identity to every $H$.

Define $\sigma(j(b,H))=\ell(H)x_1(b)$. By \eqref{eq:row-action},
$$
\sigma(j(b,H))\sigma(j(b',H')) =\ell(HH')x_1(bH'+b') =\sigma(j(bH'+b',HH')),
$$
which is precisely the matrix multiplication law in $J$. Thus $\sigma:J\to S$ is a homomorphism lifting the inclusion $J\to G$.

For $H\in\GL_{n-2}(B)$, set $\ell_0(H)=\ell(\operatorname{diag}(1,H))$. Since $\GL_{n-2}(B)=E_{n-2}(B)$, an elementary word representing $H$ uses only coordinates $3,\ldots,n$. Its evaluation is $\ell_0(H)$ by the definition of $\ell$, so this lift has support on those coordinates. For $n=4$, the word is evaluated through $\St_3(B)$ on coordinates $2,3,4$, without using a rank-two Steinberg group. For $k(a,b,H)\in K$, one consequently has
$$
\sigma(k(a,b,H))=\ell_0(H)x_2(b)x_1(a), \eqnum[eq:edge-section]
$$
where both row products range over columns $3,\ldots,n$.

Put $w_{ij}=X_{ij}(1)X_{ji}(1)X_{ij}(1)$. The Steinberg relations and characteristic two give $w_{ij}^2=1$. We next verify, inside the Steinberg presentation, that
$$
w_{12}X_{ij}(c)w_{12}^{-1} =X_{\tau(i),\tau(j)}(c). \eqnum[eq:weyl-action]
$$
Write $u=X_{12}(1)$ and $v=X_{21}(1)$, so that $w_{12}=uvu$. For $j\geq3$, the Steinberg relations imply
$$
\begin{aligned}
uX_{1j}(c)u^{-1}&=X_{1j}(c),\\
vX_{1j}(c)v^{-1}&=X_{2j}(c)X_{1j}(c),\\
u\bigl(X_{2j}(c)X_{1j}(c)\bigr)u^{-1}&=X_{2j}(c),
\end{aligned}
\qquad\text{and}\qquad
\begin{aligned}
uX_{j1}(c)u^{-1}&=X_{j2}(c)X_{j1}(c),\\
v\bigl(X_{j2}(c)X_{j1}(c)\bigr)v^{-1}&=X_{j2}(c),\\
uX_{j2}(c)u^{-1}&=X_{j2}(c).
\end{aligned}
$$
The cancellations in the last terms use characteristic two. Since $w_{12}^2=1$, the same formulas establish the identities with $1$ and $2$ interchanged. Generators $X_{ij}(c)$ with $i,j\notin\{1,2\}$ commute with both $u$ and $v$. For the opposite roots, choose $j\geq3$ and use
$$
X_{12}(c)=[X_{1j}(c),X_{j2}(1)].
$$
The preceding formulas send the right-hand side to $[X_{2j}(c),X_{j1}(1)]=X_{21}(c)$, and the reverse case follows again from $w_{12}^2=1$. Thus \eqref{eq:weyl-action} holds for every defining generator.

The element $w_{12}$ commutes with $\ell_0(H)$ and interchanges the two row products in \eqref{eq:edge-section}; those products commute with each other. Hence
$$
w_{12}\sigma(k)w_{12}^{-1}=\sigma(\eta(k))
\qquad\text{for every }k\in K.
$$
Moreover, $\sigma(h)=w_{23}$, since the chosen elementary word on the last coordinates represents $(23)$. Applying \eqref{eq:weyl-action} and its version with indices permuted to the defining words for $w_{23}$ and $w_{12}$, we obtain
$$
w_{12}w_{23}w_{12}=w_{13}=w_{23}w_{12}w_{23},
$$
so the triangle relation also lifts to $S$.

The presentation \eqref{eq:borel-presentation} now defines a homomorphism $\psi:\Pi\to S$ with $\psi|_J=\sigma$, $\psi(T)=w_{12}$, and $\phi_n\psi=p$. It remains to prove that $\psi$ is surjective. Its image contains $X_{ij}(c)$ for $i,j\geq2$ through $\ell$, and contains $X_{1j}(c)$ through the upper-row subgroup of $\sigma(J)$. By \eqref{eq:weyl-action}, the conjugates of $X_{i2}(c)$ for $i\geq3$ and of $X_{12}(c)$ by $w_{12}$ are $X_{i1}(c)$ and $X_{21}(c)$, respectively. Thus $\psi$ is surjective.

Since $p$ is an isomorphism, $s=\psi p^{-1}:G\to S$ is a surjective section of $\phi_n$. It follows that $s\phi_n=1_S$, so $s$ is the inverse of $\phi_n$.
\end{proof}

The next lemma controls the image of the kernel after one stabilization, using an additional index as in \cite[Chapter~III, proof of Theorem~5.2.1]{Weibel13}.

\begin{lemma}
\label{lem:padded-centrality}
For any unital ring $B$ and $m\geq3$, $j_m(N_m(B))$ is central in $\St_{m+1}(B)$.
\end{lemma}

\begin{proof}
For a column $a\in B^m$ and a row $b\in M_{1,m}(B)$, put
$$
u(a)=\prod_{i=1}^mX_{i,m+1}(a_i)
\qquad\text{and}\qquad
v(b)=\prod_{i=1}^mX_{m+1,i}(b_i).
$$
The factors in each product commute. The restriction of $\phi_{m+1}$ to either subgroup is injective, since its matrix image determines every coefficient. For $i,j\leq m$, the Steinberg relations imply
$$
X_{ij}(c)u(a)X_{ij}(c)^{-1}=u(t_{ij}(c)a)
\qquad\text{and}\qquad
X_{ij}(c)v(b)X_{ij}(c)^{-1}=v(bt_{ij}(c)^{-1}).
$$
Only $a_i$ changes in the first formula, by addition of $ca_j$; only $b_j$ changes in the second, by addition of $-b_i c$. Thus every word on the first $m$ indices acts on these two subgroups through its matrix image. An element of $j_m(N_m(B))$ has trivial matrix image, and hence centralizes both subgroups.

Every element of $j_m(N_m(B))$ also centralizes the generators supported on the first $m$ indices. Indeed, for distinct $i,j\leq m$,
$$
X_{ij}(c)=[X_{i,m+1}(c),X_{m+1,j}(1)].
$$
Both factors lie in the two subgroups already centralized. Therefore every element of $j_m(N_m(B))$ centralizes every defining generator of $\St_{m+1}(B)$, which proves the lemma.
\end{proof}

Over $R$, the $GE$ property reduces elementary generation to a calculation with commutators in $R^\times$.

\begin{lemma}
\label{lem:elementary-generation}
For $R=L_{\F_2}(1,2)$, one has $\GL_m(R)=E_m(R)$ for every $m\geq2$.
\end{lemma}

\begin{proof}
The algebra $R$ is $V_{1,2}$ in \cite[\S4]{AGP02}: the universal row is $(e,f)$ and the inverse column is $(e^*,f^*)^{\mathrm t}$. Theorem~4.2 of that paper shows that $R$ is purely infinite simple. The proof of its Theorem~2.3 invokes Menal--Moncasi's $GE$ argument \cite[Theorem~2.2 and the remark after Corollary~2.3]{MM84}. The latter remark applies to a simple ring in which each nonzero $x$ admits $y,z$ with $yxz=1$, and requires no regularity assumption. Thus every invertible matrix over $R$ is a product of elementary matrices and invertible diagonal matrices.

By Theorem~\ref{thm:acyclicity}, the group $R^\times$ is perfect. It therefore remains to place every diagonal factor in $E_m(R)$. In a block on two coordinates, put $W(u)=t_{12}(u)t_{21}(-u^{-1})t_{12}(u)$. Then $W(u)W(-1)=\operatorname{diag}(u,u^{-1})$ by direct multiplication. For units $a,b$,
$$
\operatorname{diag}([a,b],1) =\operatorname{diag}(a,a^{-1}) \operatorname{diag}(b,b^{-1}) \operatorname{diag}(a^{-1}b^{-1},ba).
$$
Each factor on the right has the form $\operatorname{diag}(u,u^{-1})$ and hence belongs to $E_2(R)$. Since every unit is a finite product of commutators, $\operatorname{diag}(u,1)\in E_2(R)$ for every $u\in R^\times$. Embedding this $2\times2$ calculation in any chosen pair of coordinates shows that every invertible diagonal matrix lies in $E_m(R)$ for $m\geq2$. The $GE$ decomposition now proves the lemma.
\end{proof}

\begin{theorem}
\label{thm:steinberg-isomorphism}
For $R=L_{\F_2}(1,2)$ and every $r\geq3$, the canonical map $\phi_r:\St_r(R)\to\GL_r(R)$ is an isomorphism.
\end{theorem}

\begin{proof}
Write $S_r=\St_r(R)$, $G_r=\GL_r(R)$, and $N_r=\ker\phi_r$. By Lemma~\ref{lem:elementary-generation}, $\phi_3$ is surjective, and $S_3$ is perfect since each generator is a commutator using a third index. The five-term integral homology sequence for $1\to N_3\to S_3\to G_3\to1$ contains
$$
0=H_2(G_3;\Z)\longrightarrow N_3/[S_3,N_3] \longrightarrow H_1(S_3;\Z)=0,
$$
where $H_2(G_3;\Z)=0$ by Theorem~\ref{thm:acyclicity} and Lemma~\ref{lem:leaf-coordinates}. Hence $N_3=[S_3,N_3]$. By Lemma~\ref{lem:padded-centrality}, $j_3(N_3)$ is central in $S_4$, so
$$
j_3(N_3)=[j_3(S_3),j_3(N_3)]=1.
$$
Moreover, $G_2=E_2(R)$ and $G_3=E_3(R)$ by Lemma~\ref{lem:elementary-generation}, while $|X_4(R)|$ is simply connected by Corollary~\ref{cor:frame-simple-connectivity}. Theorem~\ref{thm:raw-criterion} therefore implies that $\phi_4$ is an isomorphism.

To pass from rank four to the other ranks, let $T_r:R^r\to R^{r+1}$ fix the first $r-1$ coordinates and send the last coordinate $x$ to $(e^*x,f^*x)^{\mathrm t}$. Let $U_r:R^{r+1}\to R^r$ fix the first $r-1$ coordinates and send the last pair $(y,z)$ to $ey+fz$. By the Leavitt identities, $U_rT_r=I_r$ and $T_rU_r=I_{r+1}$, so $\alpha_r(A)=T_rAU_r$ is a unital isomorphism of matrix rings.

For $r\geq3$, the corresponding Steinberg isomorphism $D_r:S_r\to S_{r+1}$ is given by
$$
\begin{aligned}
\eqrownum &\quad D_rX_{ij}(a)=X_{ij}(a) \quad \text{for }i,j<r,\\
\eqrownum &\quad D_rX_{ir}(a)=X_{ir}(ae)X_{i,r+1}(af) \quad \text{for }i<r,\\
\eqrownum[eq:steinberg-refinement] &\quad D_rX_{rj}(a)=X_{rj}(e^*a)X_{r+1,j}(f^*a) \quad \text{for }j<r.
\end{aligned}
$$
Apply \cite[\S4, Proposition~1]{Voronetsky20} to $M_{r+1}(R)$ with $S=\{1\}$, the specialization to ordinary groups explained immediately before Lemma~5. The standard matrix idempotents form a complete orthogonal family of full, Morita equivalent idempotents, since $E_{aa}=E_{ab}E_{bb}E_{ba}$. After merging the last two, $T_r,U_r$ identify the resulting presentation with $S_r$; the original family has $r+1\geq4$ members and defines $S_{r+1}$. The proposition therefore identifies these groups by the formulas in \eqref{eq:steinberg-refinement}. Thus $D_r$ is an isomorphism obtained by refinement, distinct from the standard stabilization $j_r$.

The refinement maps satisfy
$$
\phi_{r+1}D_r=\alpha_r\phi_r.
$$
Since both $D_r$ and $\alpha_r$ are isomorphisms, $D_r(N_r)=N_{r+1}$. Starting from $N_4=1$, iteration implies $N_r=1$ for $r>4$, and $D_3(N_3)=N_4$ implies $N_3=1$. Surjectivity of every $\phi_r$, $r\geq3$, follows from Lemma~\ref{lem:elementary-generation}.
\end{proof}

\section{Finite presentability}
\label{sec:finite-presentation}

Finite presentability now follows from Theorem~\ref{thm:steinberg-isomorphism} and the theorem of Krsti\'c and McCool. The injectivity of the comparison map is essential, since a quotient of a finitely presented group need not be finitely presented. The connection with unstable $K_2$ is discussed in \cite[\S8]{KT26}. We record a finite presentation of $R$ below; Appendix~\ref{app:explicit-presentation} specifies the resulting group presentation and the images of its generators in $G$.

\begin{theorem}\label{thm:finite-presentation}
The group $G=R^\times$ is finitely presented.
\end{theorem}

\begin{proof}
Let $F=\Z\langle e,f,s,t\rangle$ be the free associative unital ring on $e,f,s,t$, and let $J$ be the two-sided ideal generated by
$$
2,\quad se-1,\quad tf-1,\quad sf,\quad te,\quad es+ft-1. \eqnum[eq:ring-ideal-generators]
$$
There is a unital ring isomorphism $F/J\cong R$ sending $s$ to $e^*$ and $t$ to $f^*$. Indeed, the displayed relations hold in $R$, while $2=0$ makes $F/J$ an $\F_2$-algebra whose generators satisfy the Leavitt relations. The two maps supplied by these universal properties fix the four generators and are inverse. Thus $R$ is finitely presented as an associative unital ring in the sense used by Krsti\'c and McCool.

By \cite[Theorem~3]{KM99}, $\St_n(B)$ is finitely presented whenever $B$ is a finitely presented associative unital ring and $n\geq4$. Hence $\St_5(R)$ is finitely presented. Theorem~\ref{thm:steinberg-isomorphism} identifies this group with $\GL_5(R)$ through the canonical map $\phi_5$, and Lemma~\ref{lem:leaf-coordinates} identifies $\GL_5(R)$ with $G$. Therefore $G$ is finitely presented.
\end{proof}

\appendix
\section{An explicit finite presentation}
\label{app:explicit-presentation}

Retain $F=\Z\langle e,f,s,t\rangle$ and the two-sided ideal $J$ generated by the six elements in \eqref{eq:ring-ideal-generators}, so that $F/J\cong R$. We first identify the kernel of the natural map $\St_n(F)\to\St_n(F/J)$ to determine the group relations corresponding to these six ring relations. The following argument is used in \cite[pp.~178--179]{KM99}.

For $n\geq3$, let $K$ be the normal closure in $\St_n(F)$ of the elements $X_{ij}(y)$, where $i\ne j$ and $y$ ranges over the six generators of $J$ in \eqref{eq:ring-ideal-generators}. Put
$$
I=\{a\in F:X_{ij}(a)\in K\text{ for every }i\ne j\}.
$$
The additivity relation makes $I$ an additive subgroup. To see that it is a two-sided ideal, let $a\in I$, $b\in F$, and choose $k\notin\{i,j\}$. Then
$$
X_{ij}(ba)=[X_{ik}(b),X_{kj}(a)]\qquad\text{and}\qquad X_{ij}(ab)=[X_{ik}(a),X_{kj}(b)],
$$
so $ba,ab\in I$ by normality of $K$. Thus $J\subseteq I$.

Let $\pi:\St_n(F)\to\St_n(F/J)$ be the natural map. Every normal generator of $K$ lies in $\ker\pi$, so $\pi$ induces a homomorphism $\bar\pi:\St_n(F)/K\to\St_n(F/J)$. Conversely, since $J\subseteq I$, the class of $X_{ij}(a)$ in $\St_n(F)/K$ depends only on $a+J$: if $a-b\in J$, then
$$
X_{ij}(a)X_{ij}(b)^{-1}=X_{ij}(a-b)\in K.
$$
These classes satisfy the Steinberg relations over $F/J$ and therefore define a homomorphism inverse to $\bar\pi$. Consequently,
$$
\St_n(F)/K\cong\St_n(F/J). \eqnum[eq:steinberg-ring-quotient]
$$

We use the finite presentation constructed by Krsti\'c and McCool in \cite[\S3]{KM99}. Let $\mathcal T=\{e,f,s,t\}$ be an alphabet, let $\mathcal T^*$ be its free monoid, and denote the empty word by $\varepsilon$. Word multiplication is concatenation. Set $m=194$ and
$$
\mathcal W=\{u\in\mathcal T^*:1\leq |u|\leq m\}\qquad\text{and}\qquad \mathcal W^0=\mathcal W\cup\{\varepsilon\}.
$$
For $d=2,3,4$, let $\mathcal I_d$ be the set of ordered $d$-tuples of pairwise distinct elements of $\{1,2,3,4,5\}$. Define
$$
\begin{aligned}
\mathcal A&=\{(u,v)\in\mathcal W^2:|u|+|v|\leq m\},\\
 \mathcal B&=\{(u,v)\in\mathcal W^2:|u|+|v|\leq m+3\},\\
 \mathcal C&=\{(u,v)\in\mathcal W^2:|u|<m\}\qquad\text{and}\qquad\mathcal D=\mathcal C,\\
 \mathcal P&=\{(\varepsilon,\varepsilon)\}
       \cup\{(\varepsilon,a),(a,\varepsilon):a\in\mathcal T\}.
\end{aligned}
$$
Membership in $\mathcal W^2$ requires each word in $\mathcal B$ to have length at most $m$, even though their total length is allowed to be at most $m+3$. Likewise, the second word in $\mathcal C$ or $\mathcal D$ lies in $\mathcal W$, while the first satisfies the strict inequality $|u|<m$. The set $\mathcal P$ consists of nine ordered pairs.

Take the generators
$$
x_{ij}(u)\qquad\text{for }(i,j)\in\mathcal I_2\text{ and }u\in\mathcal W^0. \eqnum[eq:finite-generators]
$$
The empty word $\varepsilon$ evaluates to the coefficient $1$; the symbol $x_{ij}(\varepsilon)$ is a generator, not the group identity. Impose the following five families for every index tuple and coefficient pair in the indicated sets:
$$
\begin{aligned}
\eqrowtag{A}{eq:finite-A} &\quad [x_{ij}(u),x_{jk}(v)]=x_{ik}(uv)\\
&\quad \text{for }(i,j,k)\in\mathcal I_3\text{ and }(u,v)\in\mathcal A\cup\mathcal P;\\
\eqrowtag{B}{eq:finite-B} &\quad [x_{ij}(u),x_{kl}(v)]=1\\
&\quad \text{for }(i,j,k,l)\in\mathcal I_4\text{ and }(u,v)\in\mathcal B\cup\mathcal P;\\
\eqrowtag{$C_\ell$}{eq:finite-C-left} &\quad [x_{ik}(u),x_{jk}(v)]=1\\
&\quad \text{for }(i,j,k)\in\mathcal I_3\text{ and }(u,v)\in\mathcal C\cup\mathcal P;\\
\eqrowtag{$C_r$}{eq:finite-C-right} &\quad [x_{ki}(u),x_{kj}(v)]=1\\
&\quad \text{for }(i,j,k)\in\mathcal I_3\text{ and }(u,v)\in\mathcal C\cup\mathcal P;\\
\eqrowtag{D}{eq:finite-D} &\quad [x_{ij}(u),x_{ij}(v)]=1\\
&\quad \text{for }(i,j)\in\mathcal I_2\text{ and }(u,v)\in\mathcal D\cup\mathcal P.
\end{aligned}
$$
Thus every family includes $(\varepsilon,\varepsilon)$ and both $(\varepsilon,a)$ and $(a,\varepsilon)$ for each $a\in\mathcal T$. In \eqref{eq:finite-A}, the word $uv$ has length at most $m$, or at most one when $(u,v)\in\mathcal P$, so the generator on the right is defined.

For each $(i,j)\in\mathcal I_2$, impose also
$$
\begin{aligned}
\eqrowtag{$L_1$}{eq:finite-L1} &\quad x_{ij}(\varepsilon)^2=1,\\
\eqrowtag{$L_2$}{eq:finite-L2} &\quad x_{ij}(se)x_{ij}(\varepsilon)^{-1}=1,\\
\eqrowtag{$L_3$}{eq:finite-L3} &\quad x_{ij}(tf)x_{ij}(\varepsilon)^{-1}=1,\\
\eqrowtag{$L_4$}{eq:finite-L4} &\quad x_{ij}(sf)=1,\\
\eqrowtag{$L_5$}{eq:finite-L5} &\quad x_{ij}(te)=1,\\
\eqrowtag{$L_6$}{eq:finite-L6} &\quad x_{ij}(es)x_{ij}(ft)x_{ij}(\varepsilon)^{-1}=1.
\end{aligned}
$$
Since $|\mathcal I_2|=5\cdot4=20$, these six families contribute $6\cdot20=120$ labelled relations. The inverse signs are retained because the preceding presentation is over $\Z$; characteristic two is imposed only after \eqref{eq:finite-L1} is added.

\begin{proposition}\label{prop:explicit-presentation}
Let $\Pi$ be the group with generators \eqref{eq:finite-generators} and relations \eqref{eq:finite-A}--\eqref{eq:finite-D} and \eqref{eq:finite-L1}--\eqref{eq:finite-L6}, with exactly the finite index sets specified above. Put
$$
(\alpha_1,\alpha_2,\alpha_3,\alpha_4,\alpha_5) =(e,fe,f^2e,f^3e,f^4).
$$
For $u\in\mathcal T^*$, write $\bar u$ for its value in $R$ under $s\mapsto e^*$, $t\mapsto f^*$ and $\varepsilon\mapsto1$. Then
$$
x_{ij}(u)\longmapsto 1+\alpha_i\bar u\alpha_j^* \eqnum[eq:finite-evaluation]
$$
extends to an isomorphism $\Pi\cong G$.
\end{proposition}

\begin{proof}
The construction in \cite[Section~3, pp.~179--182]{KM99} identifies the group defined by \eqref{eq:finite-A}--\eqref{eq:finite-D} with $\St_5(F)$, taking $x_{ij}(u)$ to $X_{ij}(u)$. Its sufficient bound is $m>3|\mathcal T|^3+1$, which holds because $194>193$. The relations indexed by nonempty words use $\mathcal A$, $\mathcal B$, $\mathcal C$, and $\mathcal D$. For the generators $x_{ij}(\varepsilon)$, the pairs in $\mathcal P$ impose all instances of these relation types having at least one coefficient word equal to $\varepsilon$ and total length at most one. The induction there establishes the remaining relations with at least one coefficient equal to $1$.

In $\St_5(F)$, additivity identifies \eqref{eq:finite-L1}--\eqref{eq:finite-L6} with $X_{ij}(y)=1$ for the six polynomials in \eqref{eq:ring-ideal-generators}. By \eqref{eq:steinberg-ring-quotient}, there is an isomorphism
$$
\kappa\colon\Pi\xrightarrow{\ \cong\ }\St_5(R),\qquad x_{ij}(u)\longmapsto X_{ij}(\bar u).
$$

Let $L=(\alpha_1,\ldots,\alpha_5)$ and let $L^*$ be the column of the words $\alpha_i^*$. Distinct words in the displayed tuple are prefix-incomparable, so $\alpha_i^*\alpha_j=\delta_{ij}$. The tuple is complete because repeated substitution of $ff^*=1-ee^*$ shows that
$$
\sum_{a=0}^3 f^aee^*(f^*)^a+f^4(f^*)^4=1.
$$
Thus $LL^*=1$ and $L^*L=I_5$, and the mutually inverse unital ring maps are
$$
\Theta(A)=LAL^*\qquad\text{and}\qquad \Theta^{-1}(a)=L^*aL=(\alpha_i^*a\alpha_j)_{i,j}. \eqnum[eq:five-leaf-coordinates]
$$
The composite $\Theta\phi_5\kappa$ sends each generator to \eqref{eq:finite-evaluation}. All three maps are isomorphisms, the middle one by Theorem~\ref{thm:steinberg-isomorphism}, and this proves the proposition. Moreover,
$$
(\alpha_i\bar u\alpha_j^*)^2=0\qquad\text{for }i\ne j,
$$
so the displayed generator image has two-sided inverse $1-\alpha_i\bar u\alpha_j^*$. In $R$, this inverse is the generator image itself.
\end{proof}

The inverse isomorphism can also be expressed in the given generators. Given $g\in G$, form $A=(\alpha_i^*g\alpha_j)_{i,j}$ and choose an elementary factorization
$$
A=\prod_{\nu=1}^q(I_5+a_\nu E_{i_\nu j_\nu}).
$$
Such a factorization exists because $\GL_5(R)=E_5(R)$. Lift each $a_\nu$ to a finite integer linear combination of words in $F$. Represent a coefficient sum by a product in the corresponding root subgroup, and an integer multiple by a power. For a word $u$ of length at most $m$, use $x_{ij}(u)$; the empty word uses $x_{ij}(\varepsilon)$.

For a longer word $u=va$, with $a$ its last letter, use the expression $[x_{ik}(v),x_{kj}(a)]$, where $k$ is the least index outside $\{i,j\}$ and the first entry is constructed recursively if $|v|>m$. The recursion terminates because the first word becomes shorter at each step. The resulting product represents the inverse image of $g$ in $\Pi$. Its independence of the polynomial lifts and the elementary factorization follows from the isomorphisms proved above.

The presentation has
$$
20(1+4+\cdots+4^{194})=\frac{20(4^{195}-1)}3
$$
generators. Every generator and relator is indexed by one of the finite sets specified above; no claim of minimality is made.

\section*{Funding}

Huynh Viet Khanh is supported by the Vietnam National Foundation for Science and Technology Development (NAFOSTED) under Grant No.~101.04-2025.41.

\end{document}